\documentclass[format=acmsmall, review=false, screen=true]{acmart}%{article}%

\pdfoutput=1

\usepackage{graphicx}
\usepackage{algorithmic}

\usepackage{amsmath}
\usepackage{thmtools,thm-restate}
\usepackage{amssymb}
\usepackage{amsbsy}
\usepackage{amsthm}

\usepackage{xcolor}
\usepackage{bbm}
\usepackage{caption}
\usepackage{subfig}

\usepackage{amsopn}

\usepackage[capitalise]{cleveref}

\DeclareMathOperator{\trace}{Tr}

\DeclareMathOperator{\rank}{rank}

\renewcommand{\vec}[1]{\mathbf{#1}}

\usepackage{booktabs} % For formal tables
\newtheorem{remark}{Remark}[section]
\usepackage[ruled]{algorithm2e} % For algorithms

\SetAlFnt{\small}
\SetAlCapFnt{\small}
\SetAlCapNameFnt{\small}
\SetAlCapHSkip{0pt}
\IncMargin{-\parindent}

\begin{document}
% Title portion. Note the short title for running heads
\title[Parallel Simultaneous Perturbation Optimization]{Parallel Simultaneous Perturbation Optimization}

\author{Atiye Alaeddini}
%\orcid{1234-5678-9012-3456}
%\affiliation{%
%  \institution{Institute for Disease Modeling}
%%  \streetaddress{104 Jamestown Rd}
%  \city{Bellevue}
%  \state{WA}
%  \postcode{98005}
%  \country{USA}}
%\email{aalaeddini@idmod.org}

\author{Daniel J.~Klein}
%\affiliation{%
%  \institution{Bill \& Melinda Gates Foundation}
%  \city{Seattle}
%  \state{WA}
%  \country{USA}}
%\email{daniel.klein@gatesfoundation.org}

\begin{abstract}
 Stochastic computer simulations enable users to gain new insights into complex physical systems. Optimization is a common problem in this context: users seek to find model inputs that maximize the expected value of an objective function. The objective function, however, is time-intensive to evaluate, and cannot be directly measured. Instead, the stochastic nature of the model means that individual realizations are corrupted by noise. More formally, we consider the problem of optimizing the expected value of an expensive black-box function with continuously-differentiable mean, from which observations are corrupted by Gaussian noise. We present Parallel Simultaneous Perturbation Optimization (PSPO), which extends a well-known stochastic optimization algorithm, simultaneous perturbation stochastic approximation, in several important ways. Our modifications allow the algorithm to fully take advantage of parallel computing resources, like high-performance cloud computing. The resulting PSPO algorithm takes fewer time-consuming iterations to converge, automatically chooses the step size, and can vary the error tolerance by step. Theoretical results are supported by a numerical example. To demonstrate the performance of the algorithm, we implemented the algorithm to maximize the pseudo-likelihood of a stochastic epidemiological model to data of a measles outbreak.
\end{abstract}

%
% The code below should be generated by the tool at
% http://dl.acm.org/ccs.cfm
% Please copy and paste the code instead of the example below.
%
%\begin{CCSXML}
%<ccs2012>
%<concept>
%<concept_id>10002950.10003714</concept_id>
%<concept_desc>Mathematics of computing~Mathematical analysis</concept_desc>
%<concept_significance>300</concept_significance>
%</concept>
%</ccs2012>
%\end{CCSXML}
%
%\ccsdesc[300]{Mathematics of computing~Mathematical analysis}

%
% End generated code
%

%\keywords{Stochastic Optimization, Simultaneous Perturbation, Parallel Computing, Second-Order Algorithm}

\maketitle

% The default list of authors is too long for headers.
%\renewcommand{\shortauthors}{A. Alaeddini and D. J. Klein}

\section{Introduction}\label{intro}
%%%%%%%%%%%%%%%%%%%%%%%%%%%%%%%%%%%%%%%%%%%%%%%%%%%%%%%%%%%

Stochastic optimization is of core practical importance in many fields of science and engineering. For instance, epidemiological systems involve individuals subject to exogenous disturbances, which emphasizes the need for models and methods capable of dealing with stochasticity. While many optimization methodologies exist for deterministic systems, these algorithms can give misleading results when applied to stochastic systems. Closed-form solutions do not generally exist for stochastic optimization problems, thus we seek an iterative algorithm which guarantees convergence to the local optimal point. 
%A key property of convex optimization problems is that the local optimum is the same as global optimum. The algorithm proposed here guarantees convergence to a local optimal point, which in the case of convexity, this local optimum is equal to the unique global optimum. 

%\textcolor{red}
Many approximation algorithms have been developed to solve a wide variety of deterministic or stochastic problems. The steepest descent method \cite{Nocedal99} is the most prominent iterative method for optimizing a complex objective function. The gradient-based algorithms, such as Robbins-Monro \cite{robbins1951stochastic}, Newton-Raphson \cite{froberg1969introduction}, and neural network back-propagation \cite{rumelhart1985learning}, rely on direct measurements of the gradient of the objective function with respect to the optimization parameter. But, in many cases the gradient of the loss function is not available. This is a common occurrence, for example, in complex systems, such as the optimization problems given in \cite{tsilifis2017efficient,alaeddini2017application}, the exact functional relationship between the loss function value and the parameters is not known, and the loss function is evaluated by measurements on the system or by running simulation. A review on the main areas of optimization via simulation can be found in \cite{fu1994optimization,hong2009brief,swisher2000survey}.

%\textcolor{red}{
Some algorithms have been developed specifically to optimize stochastic black-box cost functions. Of these, the Kiefer-Wolfowitz algorithm \cite{kiefer1952stochastic} is perhaps the most well known. This algorithm estimates the gradient from noisy measurements using finite differencing. Another well known stochastic optimization algorithm in the case of high dimensional problems is Simultaneous Perturbation Stochastic Approximation (SPSA), which is an approximation algorithm based on simultaneous perturbation \cite{spall1992multivariate}. Later, J.~C.~Spall presented a second-order variant of SPSA \cite{spall2000adaptive}. The SPSA algorithm is used extensively in many different areas, e.g.\, signal timing for traffic control \cite{ma2013solving}, and some large scale machine learning problems \cite{byrd2011use}. The convergence of this algorithm to the optimal value in the stochastic \emph{almost sure} sense makes it suitable in many applications.

The stochastic nature of some complex models, e.g.\, epidemiological models, means that each set of parameters maps to a distribution of outcomes, from which each sample (model run) can take several hours to obtain. In the quick-to-evaluate deterministic model setting, almost any classical optimization algorithm, such as steepest descent or Newton-Raphson, can be used. However, care must be taken when applying deterministic methods to stochastic objective functions, as the inherent noise causes unexpected behavior. In this paper, we introduce PSPO, an algorithm for optimization of stochastic objective functions. Despite many advantageous properties of the SPSA algorithm, mentioned earlier, it is a serial algorithm that evaluates the (stochastic) function a few points at a time, which results in low convergence rate. Researchers have looked at ways of enhancing the convergence of the SPSA algorithm, e.g.\, iterate averaging is an approach aimed at achieving higher convergence rate in a stochastic setting \cite{polyak1992acceleration}, using deterministic parameter perturbations instead of random perturbations \cite{bhatnagar2003two}, and more \cite{kocsis2006universal,spall2009feedback}.

%\textcolor{red}{
The PSPO algorithm, introduced in this paper, takes advantage of fundamentally-parallel resources like high-performance cloud-based computing. Our method is appropriate for problems with very noisy gradients. We also derive a relationship between the number of parallel rounds of computation and error tolerance of the gradient for each iteration. The main contribution of this paper is introducing a stochastic optimization algorithm which can be easily implemented on parallel computers, and to provide the minimum number of parallel computers in order to have an upper-bound on the error for each iteration. We demonstrate that PSPO works well in practice and compares favorably to the conventional simultaneous perturbation optimization algorithm (SPSA). The rest of the paper is structured as follows. In \cref{sec:prelim}, we give a review of the simultaneous perturbation optimization algorithm. The Parallel Simultaneous Perturbation Optimization (PSPO) algorithm is presented in \cref{sec:SPSApp}. The numerical simulations are given in \cref{sec:sims}, and \cref{sec:conclusion} concludes the paper.

%%%%%%%%%%%%%%%%%%%%%%%%%%%%%%%%%%%%%%%%%%%%%%%%%%%%%%%%%%
\section{Preliminaries} \label{sec:prelim}
 %%%%%%%%%%%%%%%%%%%%%%%%%%%%%%%%%%%%%%%%%%%%%%%%%%%%%%%%%%

Consider the problem of minimizing a cost function $L(\theta): \mathbb{R}^p \rightarrow \mathbb{R}$. Spall \cite{spall1992multivariate,spall2000adaptive,spall2009feedback} presented an efficient stochastic algorithm called Simultaneous Perturbation Stochastic Approximation (SPSA). The SPSA algorithm estimates the gradient and the Hessian matrix by finite difference. This algorithm basically consists of two parallel recursions for estimating the optimization parameter, $\theta$, and the Hessian matrix, $H(\theta)$. The first recursion is a stochastic equivalence of the Newton-Raphson algorithm, and the second one estimates the Hessian matrix. These two recursions are
\begin{equation} \label{spsa2}
\begin{aligned}
	&\theta_{k+1} = \theta_{k} - a_k \left( \Pi_{\mathcal{P}}  (\bar{H}_k)  \right)^{-1} \hat{\vec{g}}_k(\theta_k) \,, \\
	&\bar{H}_k =  \frac{k}{k+1} \bar{H}_{k-1} + \frac{1}{k+1} \hat{H}_{k} \,, %\left( \hat{H}_{k} - \Psi_k (\Pi_{\mathcal{P}}  (\bar{H}_k) ) \right)
\end{aligned}
\end{equation}
where $a_k$ is a positive scalar factor, $\mathcal{P}$ denotes the set of all positive definite matrices, and $\Pi_{\mathcal{P}} (\cdot)$ is the projection into the admissible set $\mathcal{P}$. Here, $\hat{\vec{g}}_k$ and $\hat{H}_{k}$ are the estimated gradient and Hessian at iteration $k$. The SPSA approach for estimating the $\hat{\vec{g}}_k(\theta)$ and $\hat{H}_k(\theta)$ follows.

Let $\Delta_k \in \mathbb{R}^p$ be vectors of $p$ mutually independent zero-mean random variables satisfying the condition of $E\{\Delta_k^{-1} \}$ be bounded. An admissible distribution is a Bernoulli $\pm 1$ distribution. The positive scalars $c_k$ are chosen such that they usually get smaller as $k$ gets larger. The two-sided estimate of the gradient at iteration $k$ is given by:
\begin{equation} \label{est_g_1st}
	\hat{\vec{g}}_k(\theta_k) =  \frac{y_k^{(+)} - y_k^{(-)} }{2c_k} \Delta_k^{-1} \,, 
\end{equation}
where $y_k^{(\pm)}$ are the noisy measurements of the cost function at $\theta_k \pm c_k \Delta_k$. Note that in \eqref{est_g_1st}, $\Delta_k^{-1}$ is the element-wise inverse of $\Delta_k$. In the case of second order SPSA, it is suggested to use a one-sided gradient approximation, given by:
\begin{equation} \label{est_g_1sided}
	\hat{\vec{g}}_k(\theta_k) =  \frac{y_k^{(+)} - y_k }{c_k} \Delta_k^{-1} \,.
\end{equation}

Now let $\tilde{\Delta}_k \in \mathbb{R}^p$ be vectors of $p$ mutually independent zero-mean random variables satisfying the same condition of $\Delta_k$. The positive scalars $\tilde{c}_k$ are also chosen such that they get smaller as $k$ gets larger. The numerical value of $\tilde{c}_k$ is suggested to be chosen smaller than $c_k$. The following formula gives a per-iteration estimate of the Hessian matrix.
\begin{equation} \label{est_H}
	\hat{H}_k =  \frac{1}{2} \left[ \frac{\delta G_k}{2\tilde{c}_k} \tilde{\Delta}_k^{-1} + \left( \frac{\delta G_k}{2\tilde{c}_k} \tilde{\Delta}_k^{-1} \right)^T \right] \,, 
\end{equation}
where 
\begin{equation} 
	\delta G_k = \hat{\vec{g}}_k(\theta_k + \tilde{c}_k \tilde{\Delta}_k ) - \hat{\vec{g}}_k(\theta_k - \tilde{c}_k \tilde{\Delta}_k ) \,.
\end{equation}

The practical implementation details of the algorithm is presented in \cite{spall2000adaptive}.

%%%%%%%%%%%%%%%%%%%%%%%%%%%%%%%%%%%%%%%%%%%%%%%%%%%%%%%%%%
\section{Parallel Simultaneous Perturbation Optimization: PSPO}  \label{sec:SPSApp}
%%%%%%%%%%%%%%%%%%%%%%%%%%%%%%%%%%%%%%%%%%%%%%%%%%%%%%%%%%

While SPSA is an efficient optimization algorithm for high dimensional problems, it requires many iterations to converge, particularly for high-noise problems. In the case of high uncertain problems, the computation of the gradient and Hessian from noisy objective function evaluations benefits from more evaluations.  The idea here is using parallel computing of the gradient and Hessian with different perturbation vectors, and use the estimated gradient and Hessian in a conjugate gradient type optimization algorithm.

\subsection{Gradient Estimation}

The gradient vector at each iteration can be obtained by performing multiple evaluations of \eqref{est_g_1sided} for different values of perturbation. Let $\vec{g}$ be the gradient at a point $\theta$, and $\vec{g}^i$ represents the directional gradient in $\Delta_i$ direction at this point. Then we know that:
$$ \vec{g}^i = \frac{\left< \vec{g}(\theta), \Delta_i \right>}{\| \Delta_i \|^2} \Delta_i \,.$$
Let $\hat{\vec{g}}^i$ be the estimated gradient using the one-sided gradient approximation \eqref{est_g_1sided} and perturbation vector $\Delta_i$. If $\Delta_i$ is a Bernoulli $\pm 1$ vector, we have
$$ \hat{\vec{g}}^i \approx p \vec{g}^i = \left< \vec{g}(\theta), \Delta_i \right> \Delta_i \approx \left< \hat{\vec{g}}(\theta), \Delta_i \right> \Delta_i \,.$$
Then,
$$ \hat{\vec{g}}^T \Delta_i \approx \frac{\delta f_i}{c} \,,$$
$$ \delta f_i = L(\theta+c\Delta_i) - L(\theta)\,,$$
and then,
\begin{equation} \label{gradLinEq} 
\hat{\vec{g}}^T \Delta \approx \frac{1}{c} \begin{bmatrix} \delta f_1& \delta f_2& \cdots& \delta f_M \end{bmatrix}\,,
\end{equation}
where,
$$ \Delta = \begin{bmatrix} \Delta_1& \Delta_2& \cdots& \Delta_M \end{bmatrix}\,.$$
If $\{\Delta_i\}_{1\leq i\leq M}$ span $\mathbb{R}^p$, then $\Delta \Delta^T$ is invertible. So, the least square estimation of $\hat{\vec{g}}$ is given by:
\begin{equation} \label{MLS_grad}
\hat{\vec{g}} \approx \frac{1}{c} \left( \Delta \Delta^T \right)^{-1} \Delta \begin{bmatrix} \delta f_1 & \delta f_2 & \hdots & \delta f_M \end{bmatrix}^T\,.
\end{equation}

On the other hand, if $M<p$, then \eqref{gradLinEq} is an under-determined system which has non-unique solutions. Then, the solution of the minimum Euclidean norm, $\displaystyle \|\hat{\vec{g}}\|_{2}$, among all solutions is given by:
\begin{equation} \label{MN_grad}
\hat{\vec{g}} \approx \frac{1}{c} \Delta \left( \Delta^T \Delta \right)^{-1} \begin{bmatrix} \delta f_1 & \delta f_2 & \hdots & \delta f_M \end{bmatrix}^T\,.
\end{equation}

The Parallel Simultaneous Perturbation (PSP) algorithm for estimating the gradient at a given point is given in Algorithm \ref{sp_pp}. In order to improve the efficiency of this algorithm, we can compute $L(\vec{\theta})$ once, out of the loop. Doing this, we need $M+1$ function evaluations for estimating the gradient.

\begin{algorithm}
\begin{algorithmic}
\STATE{Inputs: given point $\theta$, perturbation size $c$, and \# of parallel rounds $M$}
\STATE{Randomly sample $\Delta_0$ from $\{\pm 1\}$ binary distribution}
\STATE{Initialize computation round counter $i:=1$ and spare counter $j:=1$}
\WHILE{$i \leq M$}
 	\STATE{$ \Delta_i := (I-2\vec{e}_j \vec{e}_j^T) \Delta_0$}
	\STATE{$ \delta f_i(\theta) := f(\theta + c \Delta_i) - f(\theta) $}
%	\STATE{\textbf{if} $(i\mod p) = 0$ \textbf{then} reset $j := 0$ and re-sample $\Delta_0$ from $\{\pm 1\}$ binary distribution}
	\IF{$(i\mod p) = 0$}
      		\STATE{Reset $j := 0$ and re-sample $\Delta_0$ from $\{\pm 1\}$ binary distribution}
	\ENDIF
	\STATE{$i := i+1$; $j := j+1$}
\ENDWHILE
\STATE{$ \Delta := \begin{bmatrix} \Delta_1& \Delta_2& \hdots & \Delta_M\end{bmatrix}$}
%\STATE{\textbf{if} $M\geq p$ \textbf{then} $ \hat{\vec{g}}(\theta) :=\frac{1}{c} \left(\Delta \Delta^T \right)^{-1} \Delta \begin{bmatrix} \delta f_1 & \delta f_2 & \hdots & \delta f_M\end{bmatrix}^T$}
%\STATE{\ \ \ \ \ \ \ \ \ \ \ \ \ \ \textbf{else} $ \hat{\vec{g}}(\theta) :=\frac{1}{c} \Delta \left(\Delta^T \Delta \right)^{-1} \begin{bmatrix} \delta f_1 & \delta f_2 & \hdots & \delta f_M\end{bmatrix}^T$}
\IF{$M\geq p$}
      		\STATE $ \hat{\vec{g}}(\theta) :=\frac{1}{c} \left(\Delta \Delta^T \right)^{-1} \Delta \begin{bmatrix} \delta f_1 & \delta f_2 & \hdots & \delta f_M\end{bmatrix}^T$
\ELSE
		\STATE $ \hat{\vec{g}}(\theta) :=\frac{1}{c} \Delta \left(\Delta^T \Delta \right)^{-1} \begin{bmatrix} \delta f_1 & \delta f_2 & \hdots & \delta f_M\end{bmatrix}^T$
\ENDIF
\RETURN $\hat{\vec{g}}(\theta)$
\end{algorithmic}
% \vspace{5mm}
 \caption{Parallel Simultaneous Perturbation (PSP) gradient estimation.} \label{sp_pp}
\end{algorithm}

In Algorithm \ref{sp_pp}, we suggest an efficient strategy to choose independent $\Delta$ vectors. First, we sample $\Delta_0$ from a Bernoulli $\pm 1$ distribution. In round $j$, $1 \leq j \leq p$, we need to switch the sign of the $j$th element in $\Delta_{0}$ to generate $\Delta_{j}$. The same procedure repeats for the next $p$ rounds of computations.

\begin{lemma}
Given the suggested procedure in Algorithm \ref{sp_pp} for choosing $\Delta_i$ vectors, the vectors $\{\Delta_j\}_{1\leq j\leq p}$ spans $\mathbb{R}^p$ for all $p\neq 2$.
\end{lemma}

\begin{proof}
The $\{\Delta_j\}_{1\leq j\leq p}$ vectors span $\mathbb{R}^p$ if and only if $\Delta^{(p)} = \begin{bmatrix} \Delta_1& \Delta_2& \cdots& \Delta_p \end{bmatrix}$ is full rank.
For $p=1$, it is trivial that $\rank(\Delta^{(p)})=p$. For $p=3$, let assume 
$$ \Delta_0 = \begin{bmatrix} a_1 \\ a_2 \\ a_3 \end{bmatrix}\,,$$
then
$$ \Delta^{(3)} = \begin{bmatrix} -a_1 & a_1 & a_1 \\ a_2 &-a_2 & a_2\\ a_3&a_3 & -a_3 \end{bmatrix}\,.$$
We know that $\det{\Delta^{(3)}}=4a_1a_2a_3 \neq 0$, for all $a_i\neq 0$. Since $a_i=\pm1$, $\Delta^{(3)}$ is full rank. Now, assume we know that $\rank(\Delta^{(k)})=k$. We need to prove that $\rank(\Delta^{(k+1)})=k+1$. From the procedure in Algorithm \ref{sp_pp}, we have
$$ \Delta^{(k+1)} = \begin{bmatrix} -a_1 & \hdots & a_1& a_1 \\ \vdots &\ddots & \vdots & \vdots \\ a_k &\hdots & -a_k & a_k \\ a_{k+1}& \hdots & a_{k+1}& -a_{k+1} \end{bmatrix}\,.$$
Since adding a row to a matrix does not change the rank of that matrix, it is clear that 
$$ \rank(\Delta^{\prime(k)})= \rank \left( \begin{bmatrix} -a_1 & \hdots & a_1 \\ \vdots &\ddots & \vdots \\ a_k&\hdots & -a_k\\ a_{k+1}& \hdots & a_{k+1} \end{bmatrix} \right) = k \,.$$
Now, we need to prove that $\Delta_{k+1}$ (the last column of $\Delta^{(k+1)}$) is not in the span of the columns of the matrix $\Delta^{\prime(k)}$. Let $\vec{c}_i$ represent the columns of $\Delta^{\prime(k)}$. If $\Delta_{k+1}$ is not independent of the columns, $\vec{c}_i$, then there should exists a nonzero vector, say $\vec{\alpha}\neq \vec{0}$, such that 
$$ \sum_{i=1}^k \alpha_i c_i + \alpha_{k+1} \Delta_{k+1} = \vec{0}\,,$$
which results in
$$ \sum_{i=1}^{k+1} \alpha_i - 2 \alpha_j= 0, \ \ j=1,\cdots,k+1 $$
%$$ \sum_{i=1}^k \alpha_i = -1 \,.$$
It is easy to show that there is not any real valued $\vec{\alpha}$ which satisfies these conditions. Thus, $\Delta_{k+1}$ is independent of the columns of $\Delta^{(k+1)}$, and adding the $\Delta_{k+1}$ column, which generates $\Delta^{(k+1)}$, increases the rank of the matrix by one. Therefore, $\rank \Delta^{(k+1)} = k+1$, which proves the lemma by induction. 
%By substituting $\Delta_i = (I-2\vec{e}_i \vec{e}_i^T) \Delta_0$, we have
%$$ \left( \Delta_0 \vec{1}^T -2 \Delta_0^T \vec{e}_i I \right) \vec{\alpha}= \vec{0}\,.$$
%The matrix $\left( \Delta_0 \vec{1}^T -2 \Delta_0^T \vec{e}_i I \right)$ is full rank unless $\exists i, s.t. \sum_j \Delta_{0_j} =2 \Delta_{0_i}$. Since we assumed that $\mean{\Delta_0}\neq \frac{\pm 2}{p}$, then $\forall i,  \sum_j \Delta_{0_j} \neq \pm 2$. So, $\left( \Delta_0 \vec{1}^T -2 \Delta_0^T \vec{e}_i I \right)$ should be full rank. Therefore, $\vec{\alpha} = \vec{0}$ which contradicts with our assumption. So the assumption of $\vec{\alpha} \neq \vec{0}$ must be false which proves the lemma.
\end{proof}

\begin{lemma} \label{lem_trBound}
Given the suggested procedure for choosing $\Delta$ vectors, if $\Delta_0 = \vec{1}$, where $\vec{1}$ is a vector that all its elements are one, then for all $p\neq 2$ and $M \geq 4$, then $\trace (\Delta \Delta^T)^{-1} \leq \frac{p}{4}$ and $\trace (\Delta^T \Delta)^{-1} \leq \frac{p}{4}$.
\end{lemma}

\begin{proof}
Using $\Delta_0 = \vec{1}$, we have $\Delta_i = \vec{1}-2\vec{e}_i$. Since $(\Delta \Delta^T)$ is a symmetric full rank matrix, we can write it as its eigenvalues-decomposition, given by:
$$ \Delta \Delta^T = Q \Lambda Q^T\,,$$
where,
$$ Q = \begin{bmatrix}  \vec{v}_1 &  \vec{v}_2 & \hdots & \vec{v}_p \end{bmatrix}\,,$$
and $\vec{v}_i$ is the eigenvector corresponding to eigenvalue $\lambda_i$. Then,
$$ \left(\Delta \Delta^T \right)^{-1}= Q \Lambda^{-1} Q^T\,.$$
We know
\begin{equation}
\begin{aligned}
p &= \trace(\left(\Delta \Delta^T \right) \left(\Delta \Delta^T \right)^{-1}) = \sum_{i=1}^M \Delta_i^T \left(\Delta \Delta^T \right)^{-1} \Delta_i \\
&= \sum_{i=1}^M \left[ \left(\sum_{j=1}^p \vec{v}_j^T \Delta_i \vec{v}_j^T\right) \left(\sum_{k=1}^p \frac{1}{\lambda_k}\vec{v}_k \Delta_i^T \vec{v}_k \right) \right] \\
&=\sum_{i=1}^M \left[ \sum_{j=1}^p \frac{1}{\lambda_j} \vec{v}_j^T \Delta_i \Delta_i^T \vec{v}_j \right] \\
&=\sum_{j=1}^p \frac{1}{\lambda_j} \sum_{i=1}^M \left[ \vec{v}_j^T (\vec{1}-2\vec{e}_i) (\vec{1}-2\vec{e}_i)^T \vec{v}_j \right] \\
&= (M-4) \left(\vec{1}^T \left(\Delta \Delta^T \right)^{-1} \vec{1} \right) + 4\trace \left(\Delta \Delta^T \right)^{-1} \\
&= \trace \left( \left((M-4) \vec{1}\vec{1}^T+4I\right) \left(\Delta \Delta^T \right)^{-1} \right)\,.
\end{aligned}
\end{equation}
Since $M-4\geq 0$ and $\vec{1}\vec{1}^T \succeq 0$, then $\left((M-4) \vec{1}\vec{1}^T+4I\right) \left(\Delta \Delta^T \right)^{-1} \succeq 4 \left(\Delta \Delta^T \right)^{-1}$. Therefore,
$$\trace \left( \left(\Delta \Delta^T \right)^{-1} \right) \leq \frac{p}{4} \,.$$

Proving $\trace (\Delta^T \Delta)^{-1} \leq \frac{p}{4}$ is very similar to the above proof. Due to this similarity, the proof is skipped here.
\end{proof}

Now let $y=L(\theta)$ and $y^i=L(\theta+c\Delta_i)$ represent the noisy measurements at $\theta$ and $\theta+c\Delta_i$ respectively. Assume these measurements can be expressed as normal random variables as $y \sim \mathcal{N}(\mu,\sigma^2)$ and $y^i \sim \mathcal{N}(\mu_i,\sigma^2)$. It is assumed that the noise variance is constant over the whole state space, $\theta$. 

\begin{lemma} \label{lem_expVal}
Given $y \sim \mathcal{N}(\mu,\sigma^2)$ and $y^i \sim \mathcal{N}(\mu_i,\sigma^2)$, if $M\geq p$, then the expected value of the gradient from Algorithm \ref{sp_pp} is equal to the gradient $\vec{g}$.
\end{lemma}

\begin{proof}

Based on our assumption, 
$$ \mathbb{E} \left[ \delta f_i \right] = \mu_i - \mu \,.$$
For small $c$, we have
$$ \vec{g}^T \Delta_i = \frac{\mu_i - \mu}{c}\,. $$
Then
$$ \mathbb{E} \left[ \delta f_i \right] =c \vec{g}^T \Delta_i \,.$$
If $M\geq p$, using \eqref{MLS_grad}, we have
\begin{equation} \label{lem_eq2} 
\begin{aligned} 
\mathbb{E} [\hat{\vec{g}}] &= \frac{1}{c} \left(\Delta \Delta^T \right)^{-1} \Delta \begin{bmatrix}  \mathbb{E} \left[ \delta f_1\right] &  \mathbb{E} \left[ \delta f_2 \right]& \hdots &  \mathbb{E} \left[ \delta f_M\right]\end{bmatrix}^T \\
&= \frac{1}{c} \left(\Delta \Delta^T \right)^{-1} \Delta c \begin{bmatrix}  \Delta_1^T \\  \Delta_2^T \\ \vdots \\  \Delta_M^T \end{bmatrix} \vec{g} \\
&=  \left(\Delta \Delta^T \right)^{-1}  \left(\Delta \Delta^T \right) \vec{g} = \vec{g}\,.
\end{aligned}
\end{equation}
%After substituting this in \eqref{lem_eq1}, we have
%\begin{equation} \label{lem_eq2} 
%\begin{aligned}
%\mathbb{E} [\hat{\vec{g}}] &= \frac{1}{c} \left(\Delta \Delta^T \right)^{-1} \Delta c \begin{bmatrix}  \Delta_1^T \\  \Delta_2^T \\ \vdots \\  \Delta_M^T \end{bmatrix} \vec{g} \\
%&=  \left(\Delta \Delta^T \right)^{-1}  \left(\Delta \Delta^T \right) \vec{g} = \vec{g}\,.
%\end{aligned}
%\end{equation}
\end{proof}

\begin{theorem}
If the rounds of computation, $M$, satisfies
\begin{equation} \label{Nparallel}
M \geq \max \left\{p,\frac{\sigma^2 p}{c^2\epsilon^2}\right\}\,,
\end{equation}
then the error in estimated gradient from Algorithm \ref{sp_pp} is bounded by
$$ \mathbb{E} \left[ \| \hat{\vec{g}}- \vec{g}\| \right] \leq \epsilon \,.$$
\end{theorem}

\begin{proof}
Using the assumptions of $y \sim \mathcal{N}(\mu,\sigma^2)$ and $y^i \sim \mathcal{N}(\mu_i,\sigma^2)$, the function difference $\delta f_i(\theta)$ can be expressed as:
$$\delta f_i(\theta) \sim \mathcal{N}( \mu_i-\mu,2\sigma^2)\,.$$
Using \eqref{MLS_grad} for estimating the gradient and \cref{lem_expVal}, the error in gradient, $(\hat{\vec{g}} - \vec{g})$, is a normal random variable given by:
$$ (\hat{\vec{g}} - \vec{g}) \sim \mathcal{N} (0,\Sigma) \,,$$
where
$$ \Sigma= \frac{2\sigma^2}{c^2} \left(\Delta \Delta^T \right)^{-1} \Delta \left( \left(\Delta \Delta^T \right)^{-1} \Delta \right)^T = \frac{2\sigma^2}{c^2} \left(\Delta \Delta^T \right)^{-1} \,.$$
After $M$ rounds of computation of $\hat{\vec{g}}$, Chebyshev inequality implies that
$$ \Pr \left[ \| \hat{\vec{g}} - \vec{g} \| \geq \epsilon \right] \leq \frac{\trace(\Sigma)}{M \epsilon^2} = \frac{2 \sigma^2 \trace \left(\left(\Delta \Delta^T \right)^{-1}\right)}{c^2 M \epsilon^2} \,.$$
Using the results of \cref{lem_trBound}, we have
$$ \Pr \left[ \| \hat{\vec{g}} - \vec{g} \| \geq \epsilon \right] \leq \frac{\sigma^2 p}{2c^2 M \epsilon^2} \,.$$
Therefore, if 
$$M \geq \frac{\sigma^2 p}{c^2\epsilon^2}\,,$$
then 
$$ \Pr \left[ \| \hat{\vec{g}} - \vec{g} \| \geq \epsilon \right] \leq \frac{1}{2} \,.$$
Therefore, $\displaystyle \mathbb{E}  \left[ \| \hat{\vec{g}} - \vec{g} \| \right] \leq \epsilon\,.$
\end{proof}

%\color{red}
\subsection{Hessian Estimation}

Similar to every other Newton-based algorithm, we need to estimate the Hessian (matrix of second derivatives). In this work, we are suggesting estimation of the reduced Hessian, $H_{\vec{p}}$, instead of full Hessian, $H$. The reduced Hessian, $H_{\vec{p}}$, only conveys the information of the effect of the true Hessian in a specific $\vec{p}$ direction. The reduced Hessian with respect to a given vector $\vec{p}$ satisfies
$$ \vec{p}^T H_{\vec{p}} \vec{p} = \vec{p}^T H \vec{p}\,. $$
Before introducing how we can compute the reduced Hessian, we present a method to estimate the true Hessian in the following lemma.

\begin{restatable}{lemma}{primelemma} \label{Hess_Est_lem}
Suppose $f:\mathbb{R}^n \rightarrow  \mathbb{R}$ is a function taking as input a vector $\vec{x} \in \mathbb{R}^n$. Given a set of small $n$ orthogonal vectors $\{ \vec{d}_1, \vec{d}_2, \cdots, \vec{d}_n \}$, then the true Hessian at a point $\vec{x}$ is given by
\begin{equation} \label{EstHess}
 H(\vec{x}) \approx \sum_{i=1}^n \frac{\Delta G_i}{\| \vec{d}_i \|^2} \vec{d}_i^T \,,
\end{equation}  
where
$$ \Delta G_i = \vec{g}(\vec{x}+\vec{d}_i) - \vec{g}(\vec{x}) \,.$$
\end{restatable}

The proof of this lemma can be found in the appendix (see \cref{sec:proof_of_mylemma}).

\begin{definition}
The reduced Hessian with respect to a given vector $\vec{d}$ is defined as:
\begin{equation} \label{ReducHess}
 H_{\vec{d}} (\vec{x}) = \frac{1}{2 \| \vec{d} \|^2}\left[ (\vec{g}(\vec{x}+\vec{d}) - \vec{g}(\vec{x})) \vec{d}^T +  \vec{d} (\vec{g}(\vec{x}+\vec{d}) - \vec{g}(\vec{x}))^T\right]\,. 
\end{equation}
\end{definition}

\begin{theorem}
The reduced Hessian defined in \eqref{ReducHess} satisfies the following condition:
$$ \vec{d}^T H_{\vec{d}} \vec{d} = \vec{d}^T H \vec{d}\,. $$
\end{theorem}

\begin{proof}
Suppose $\{ \vec{d}_1, \vec{d}_2, \cdots, \vec{d}_{n-1} \}$ are a set of orthogonal vectors such that $\{ \vec{d}, \vec{d}_1, \vec{d}_2, \cdots, \vec{d}_{n-1} \}$ becomes a set of $n$ orthogonal vectors which spans the entire $\mathbb{R}^n$. Then using \eqref{EstHess}, the Hessian can be estimated as:
\begin{equation}
 H(\vec{x}) \approx \frac{\vec{g}(\vec{x}+\vec{d}) - \vec{g}(\vec{x})}{\| \vec{d} \|^2} \vec{d}^T + \sum_{i=1}^{n-1} \frac{\Delta G_i}{\| \vec{d}_i \|^2} \vec{d}_i^T \,.
\end{equation}  
Since all $\vec{d}_i, i=1, \cdots, n-1$ are perpendicular to $\vec{d}$, we have
$$ \vec{d}^T H \vec{d} = \left[ \vec{g}(\vec{x}+\vec{d}) - \vec{g}(\vec{x}) \right]^T \vec{d} \,.$$
On the other hand, from the definition in \eqref{ReducHess}
$$ \vec{d}^T H_{\vec{d}} \vec{d} = \frac{2 (\vec{g}(\vec{x}+\vec{d}) - \vec{g}(\vec{x}))^T \vec{d} \| \vec{d} \|^2}{2 \| \vec{d} \|^2} = \left[ \vec{g}(\vec{x}+\vec{d}) - \vec{g}(\vec{x}) \right]^T \vec{d} = \vec{d}^T H \vec{d} \,.$$

\end{proof}

\color{black}
\subsection{PSPO Algorithm}

Algorithm \ref{conjugateSP} presents a step-by-step summary of the proposed approach in this paper. This algorithm is a conjugate gradient type algorithm which uses the PSP algorithm to estimate the gradient and Hessian at each iteration. The stopping criterion in this algorithm can be defined by the size of the gradient at the final point, the update size, or even the number of iterations.

\begin{algorithm}[!h]
\begin{algorithmic}
\STATE{Initialize $i=0$, $k=0$ and initial guess $\theta_0$}
\STATE{Initialize update direction $\vec{d} =$ null}
\STATE{Run PSP to compute $\hat{\vec{g}}_0(\theta_0)$ with $M=1$}
\STATE{$\vec{r} := -\hat{\vec{g}}_0(\theta_0)$}
\STATE{$\vec{d} := \vec{r}$}
\STATE{$\vec{g}_{new} := \hat{\vec{g}}_0(\theta_0)$}
\WHILE{stopping criteria satisfied}
	\STATE{Set \# of parallel rounds $M$ using error tolerance for this iteration}
 	\STATE{Run PSP to compute $\hat{\vec{g}}_k(\theta_k)$ using computed $M$}
	\STATE{$\displaystyle \vec{r} := -\hat{\vec{g}}_k(\theta_k)$}
	\STATE{$\displaystyle \vec{g}_{old} := \vec{g}_{new}$}
	\STATE{$\displaystyle \vec{g}_{new} :=\hat{\vec{g}}_k(\theta_k)$}
	\STATE{$\displaystyle \beta = \frac{\vec{g}_{new}^T(\vec{g}_{new}-\vec{g}_{old})}{\vec{g}_{old}^T \vec{g}_{old}}$}
	\STATE{$\displaystyle \tilde{\vec{d}} := \tilde{c} \frac{\vec{d}}{\|\vec{d}\|} + \varepsilon \mathbbm{1}_{\vec{d}=0}(\vec{d})$}
	\STATE{$\theta^+_k := \theta_k + \tilde{\vec{d}}$}
	\STATE{$\theta^-_k := \theta_k - \tilde{\vec{d}}$}
 	\STATE{Run PSP to compute $\hat{\vec{g}}(\theta^+_k)$, and $\hat{\vec{g}}(\theta^-_k)$}
	\STATE{$\displaystyle \delta G_k = \hat{\vec{g}}(\theta^+_k) - \hat{\vec{g}}(\theta^-_k)$}	
	\STATE{$\displaystyle \hat{H}_k := H_{\tilde{\vec{d}}}$ from \eqref{ReducHess}} % \frac{1}{2} \left[ \frac{\delta G_k}{2} \tilde{\vec{d}}^{-1} + \left( \frac{\delta G_k}{2} \tilde{\vec{d}}^{-1} \right)^T \right]
	\STATE{$\displaystyle \alpha :=  -\frac{\hat{\vec{g}}_k^T \vec{d}}{\vec{d}^T \hat{H}_k \vec{d}}$}
	\STATE{Update $\displaystyle \theta_{k+1} :=\theta_{k} +\alpha \vec{d}$}
	\STATE{Update direction $\displaystyle \vec{d} :=\vec{r}+\beta \vec{d}$}
	\STATE{$i := i+1$}	
	\IF{$i =p$ or $\vec{r}^T\vec{d} \leq 0$}
      		\STATE $\displaystyle \vec{d} := \vec{r}$\;
		\STATE $\displaystyle i := 0$\;
	\ENDIF	
	\STATE{$k := k+1$}
\ENDWHILE
\STATE{$\theta_{est} := \theta_k$}
\RETURN $\theta_{est}$
\end{algorithmic}
% \vspace{3mm}
 \caption{PSPO Algorithm} \label{conjugateSP}
\end{algorithm}

%\vspace{10mm}
%\textbf{Remarks}
%\begin{itemize}
%  \item PSPO does not need any determination of the step sizes. Instead, the step size is computed in a conjugate gradient setup.
\begin{remark}
The PSPO algorithm gives the chance to set the error tolerance for each step. One is able to start a coarse search at first iterations and decrease the error tolerance as the algorithm gets closer to the optimal point (by increasing the number of the computation rounds).
\end{remark}

\begin{remark}
To estimate the Hessian in PSPO, at each iteration we compute a rank-one estimate of the Hessian with respect to the update direction in that iteration, which is called a reduced Hessian in this paper. In the gradient descent algorithms, the Hessian matrix scales the update vector based on the curvature of the objective function for each feature. The idea, here, is to estimate the curvature of the objective function for each feature at each iteration, but only in the direction of the update in  that iteration. In fact, we only need to estimate the effect of the Hessian matrix in the update direction, and we are not interested on the effect of  the Hessian matrix in other directions. 
\end{remark}

\section{Simulation} \label{sec:sims}

%% ==============================      example 1     ===============================%%
\subsection{Toy Example: Minimization of A Quadratic Function}

To demonstrate the main contribution of this paper, we implemented the methodology on a variety of examples. We begin with a simple three-dimensional quadratic function to clearly illustrate the efficiency of the algorithm in the case we know the exact optimal solution. To empirically evaluate the proposed method, we investigated the efficiency of PSPO and SPSA algorithms on a stochastic optimization problem. We use the same initialization condition when comparing two optimization algorithms. We evaluate our proposed method on the following highly noisy objective function
$$ f(\vec{x}) = \left\|\vec{x}-\vec{1}\right\|_2^2+ w, \ \ \vec{x} \in \mathbb{R}^5\,,$$
where, $w$ is a Gaussian noise $w \sim \mathcal{N}(0,3^2)$. This convex objective function is suitable for comparison of different optimizers without worrying about local minimum issues. We rerun each optimization algorithm $200$ times for different initial conditions. According to \cref{PSPOvsSPSA_200Iter_maxIter100}, we found that the PSPO algorithm converges faster than SPSA on average. Note that the maximum number of iteration set to $100$ for both algorithms. The efficiency of the two algorithms is also tested through running the algorithms for a fixed initial condition many times. \cref{PSPOvsSPSA_100Iter} shows the result. In this example, the number of function evaluations per iteration for the PSPO algorithm is twice of the number of function evaluations per iteration for the SPSA algorithm. But, as it can be seen in \cref{PSPOvsSPSA_100Iter}, in case we have access to supercomputers to run the stochastic simulations in parallel, then the PSPO algorithm converges faster.
\begin{figure}[!h]
  \centering
  \subfloat[]{\includegraphics[width=.45\linewidth]{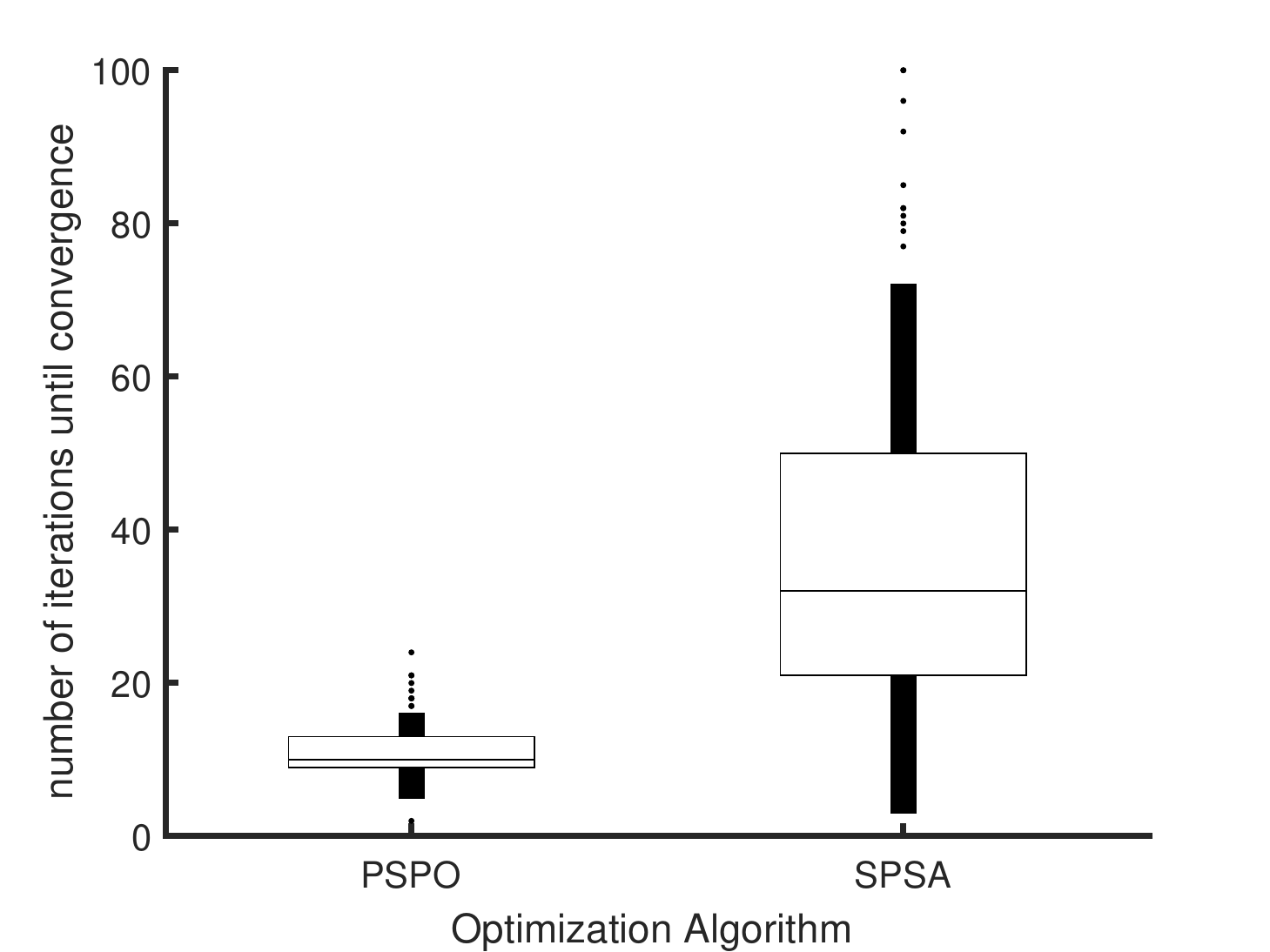} \label{PSPOvsSPSA_200Iter_maxIter100}}
	\hfill  
\subfloat[]{\includegraphics[width=.45\linewidth]{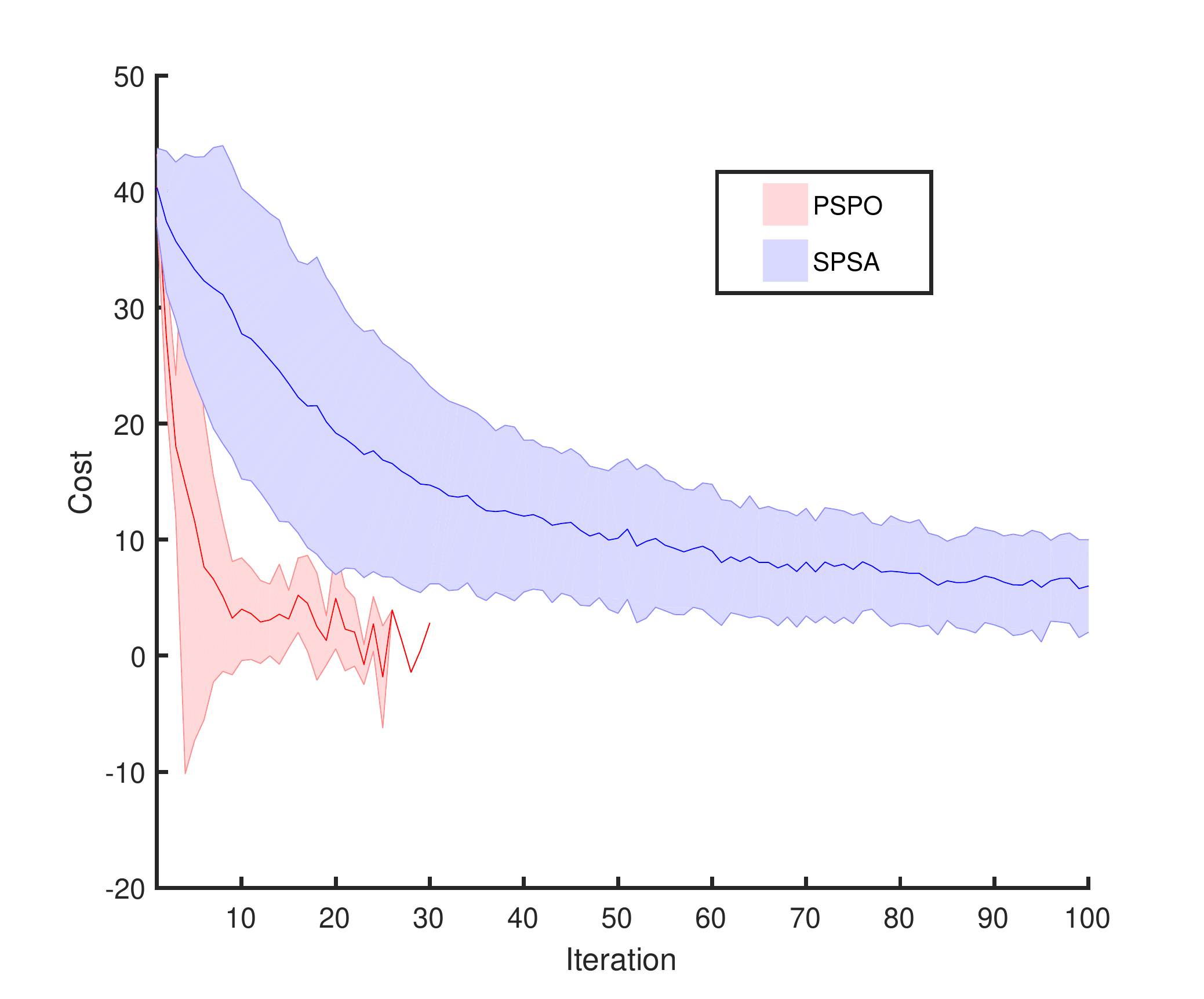} \label{PSPOvsSPSA_100Iter}}
\caption{Comparison of SPSA and PSPO on optimization of a noisy convex objective function. (a) Number of iterations to converge. (b) Incremental decrease of objective function for two algorithms.}
\label{fig:test}
\end{figure}

%\color{red}
%% ==============================      example 2     ===============================%%
\subsection{Real Application: Epidemiological Model Calibration}

To demonstrate the efficiency of the PSPO algorithm proposed in this paper, we implemented the methodology on a data set concerns a measles outbreak in a small town in Germany in 1861, which contains 188 infected individuals. This data set, called Hagelloch data set, is very popular in literature because of its completeness and depth of data. \cref{Measles1861} gives the observed clinical data. The data are obtained from \cite{meyer2014spatio}. The objective, here, is finding a good model for the given epidemic data to investigate the properties of the disease spread. To estimate the model parameters, we use maximum likelihood estimate (MLE). Details on how we can formulate this problem as an optimization problem can be found in \cite{alaeddini2017application}.
\begin{figure}[!h]
        \centering
        \includegraphics[width=.5\linewidth]{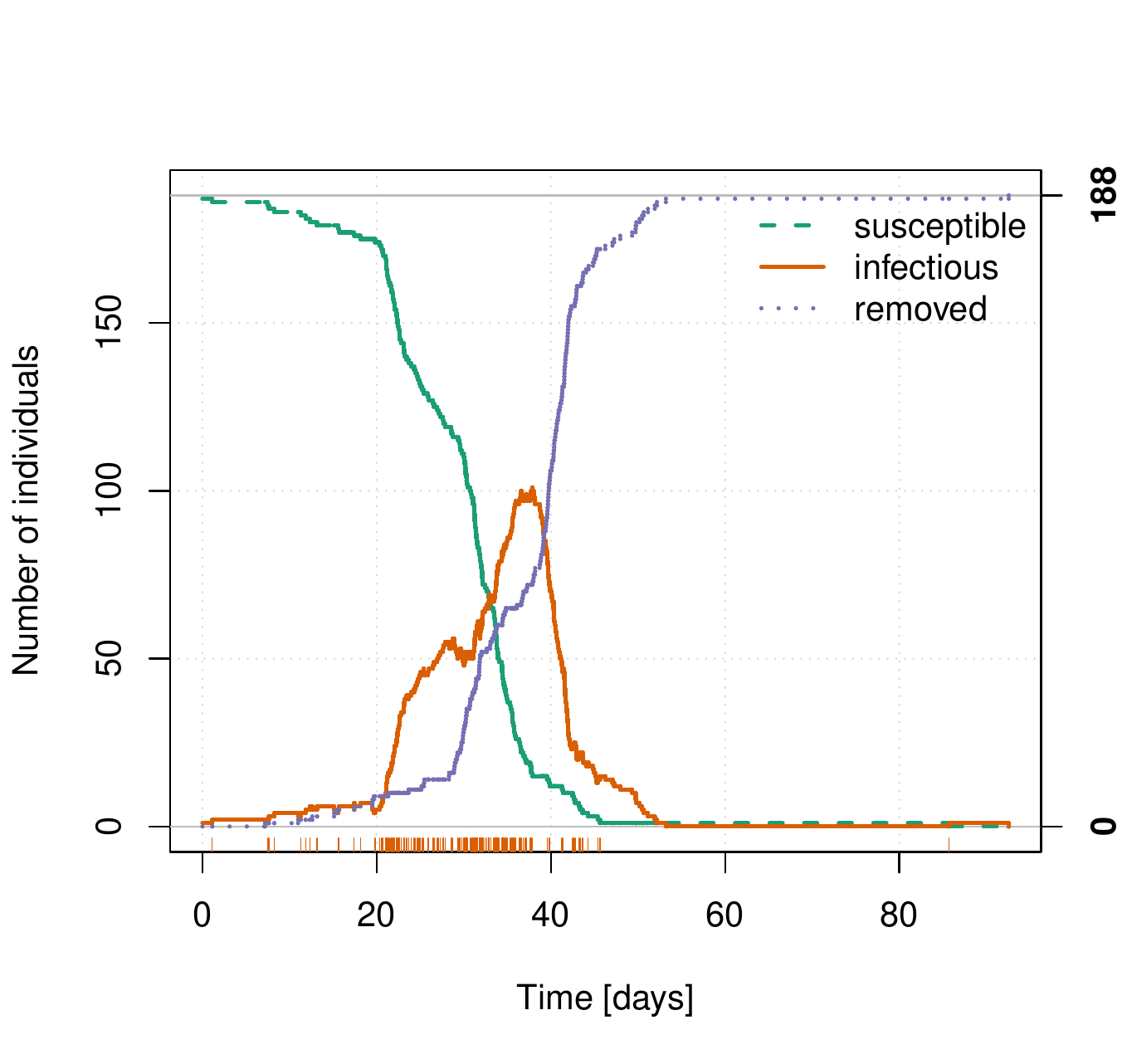}
	\vspace{-5mm}
        \caption{Measles outbreak clinical data. The orange solid line, green dashed line, and dotted blue line represent the number of infected, susceptible, and recovered people respectively.} \label{Measles1861}
\end{figure}

We compare the efficiency of the PSPO algorithm with the conventional SPSA algorithm. To compare these two algorithms, we rerun each optimization algorithms 100 times. \cref{hist19} shows the number of iterations required to converge. As it can be seen in this figure, PSPO algorithm on average needs fewer iterations compared with the conventional SPSA algorithm. The maximum number of iterations for both algorithms set to 30. Note that the number of iterations until convergence set to 30 if the algorithm does not converge after 30 iterations. In order to investigate the performance of the PSPO algorithm as the number of parallel rounds increases, we run the PSPO for this example with different number of parallel computation rounds. \cref{fig:conv_ver_M} shows the number of iterations for convergence as $M$ grows. In this example, again, the number of function evaluations per iteration for the PSPO is twice of the corresponding number for the SPSA algorithm. Having higher number of function evaluations per iteration in this example, we can still see in \cref{hist19} that the number of total function evaluations for PSPO might be less than the total number of function evaluations for SPSA. However, in this paper we do not focus on reducing the number of objective function evaluations and only focus on the faster convergence in case of having access to parallel supercomputers. 
\begin{figure}[!h]
  \centering
  \subfloat[]{\includegraphics[width=.45\linewidth]{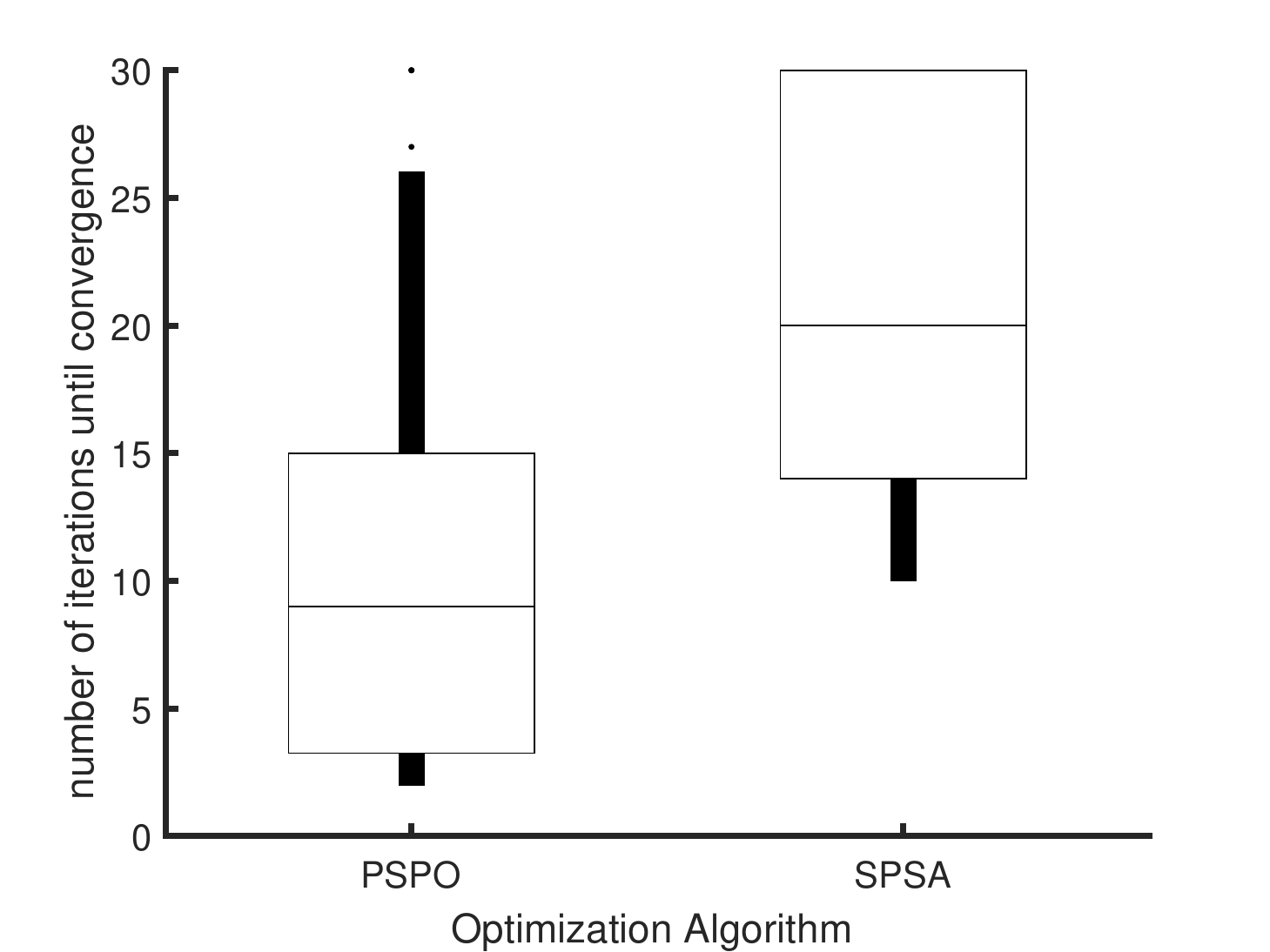} \label{hist19}}
	\hfill  
\subfloat[]{\includegraphics[width=.45\linewidth]{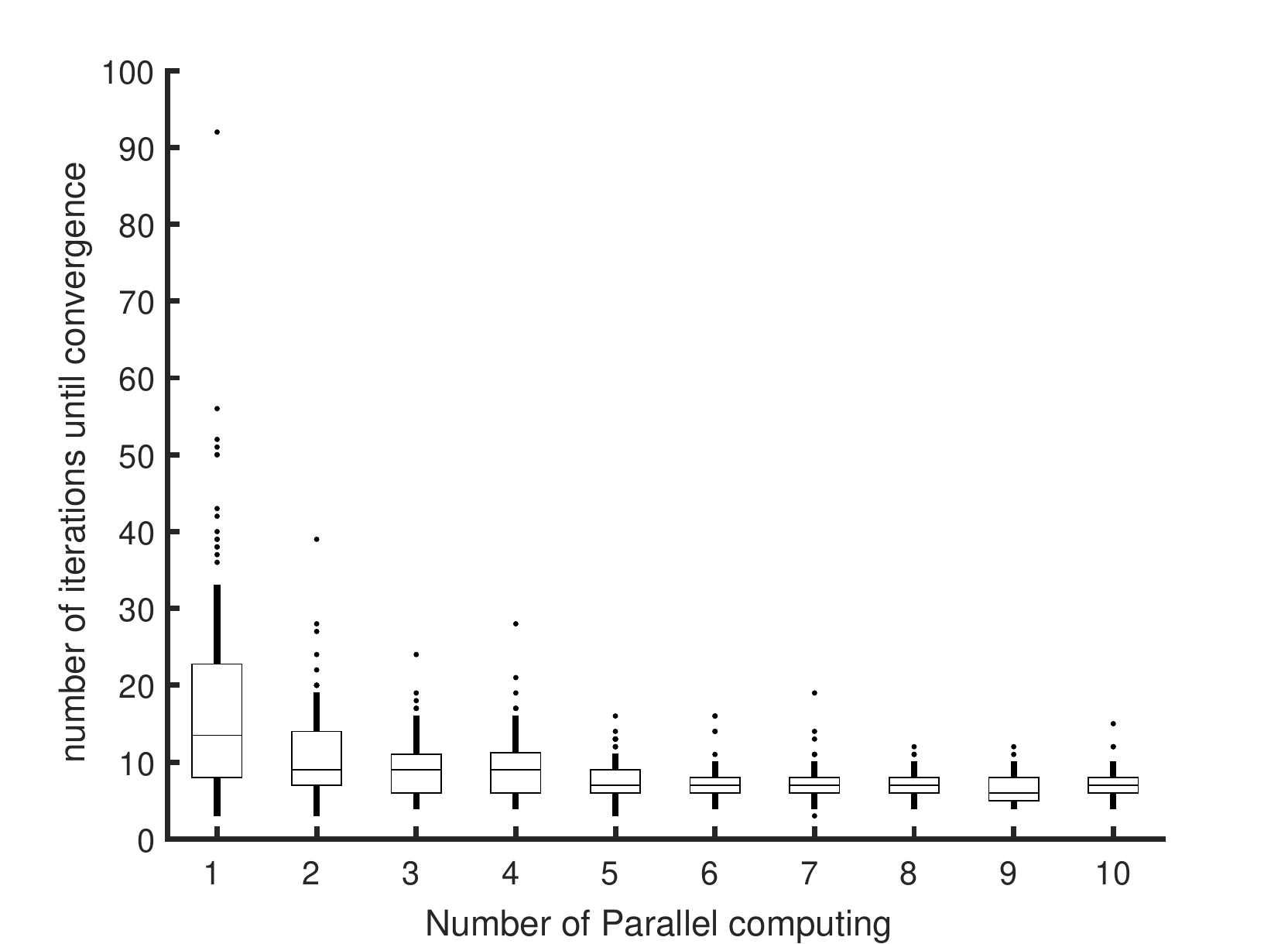} \label{fig:conv_ver_M}}
\caption{(a) Comparison of the number of iterations for convergence for conventional SPSA and PSPO. (b) Number of iterations vs. number of parallel computing rounds, $M$, in PSPO.}
\label{fig:test2}
\end{figure}

\color{black}

\section{Conclusions} \label{sec:conclusion}

In many stochastic optimization problems that noise presents, obtaining an analytical solution is hardly possible. In this paper, we described an algorithm for optimal parameter estimation in discretely observed stochastic model. We have introduced PSPO, which is a simple and computationally efficient algorithm for gradient-based optimization of stochastic objective functions. The comparison between SPSA and PSPO algorithm demonstrated the superiority of the PSPO algorithm in the number of iterations for convergence. Note that, comparing with SPSA, the PSPO algorithm with multiple parallel computations requires additional function evaluations per iteration. In case that we do not have access to powerful parallel computers, in order to prevent unnecessary long processing time per iteration, which results in long total processing time, the number of parallel computations need to be carefully computed using the level of the noise presented in the problem. We found the minimum number of parallel rounds of computation in order to have a bounded error in the estimated gradient. In some cases when we are dealing with a high SNR (signal to noise ratio) problem, a single round of computation might be enough for convergence, but in some other applications that the data are highly noisy the SPSA algorithm requires too many iterations to converge or even diverges. Our method is aimed towards highly uncertain objective functions.

%\appendix
%\section{An example appendix} 
%\lipsum[71]

\begin{acks}
The authors thank Bill and Melinda Gates for their active support of this work and their sponsorship through the Global Good Fund. This work was performed at Institute for Disease Modeling with help from colleagues, especially Philip Eckhoff.
\end{acks}

\bibliographystyle{ACM-Reference-Format}
\bibliography{citations}

\section*{Appendix} \label{sec:proof_of_mylemma}
\subsection*{Proof of \cref{Hess_Est_lem}} 

\primelemma*

\begin{proof}
We know that the Hessian matrix $H$ of $f$ is a square $n \times n$ matrix, defined as follows:
\begin{equation} \label{Hess_FinDiff}
 H = \begin{bmatrix} \frac{\partial^2 f}{\partial \vec{x}_1^2} & \frac{\partial^2 f}{\partial \vec{x}_1 \partial \vec{x}_2} & \hdots & \frac{\partial^2 f}{\partial \vec{x}_1 \partial \vec{x}_n} \\  \frac{\partial^2 f}{\partial \vec{x}_1 \partial \vec{x}_2} & \frac{\partial^2 f}{\partial \vec{x}_2^2} & \hdots & \frac{\partial^2 f}{\partial \vec{x}_2 \partial \vec{x}_n} \\ \vdots & \vdots & \ddots & \vdots \\ \frac{\partial^2 f}{\partial \vec{x}_1 \partial \vec{x}_n} & \frac{\partial^2 f}{\partial \vec{x}_2 \partial \vec{x}_n} & \hdots & \frac{\partial^2 f}{\partial \vec{x}_n^2}\end{bmatrix} = \begin{bmatrix} \frac{\partial \vec{g}_1}{\partial \vec{x}}  \\  \frac{\partial \vec{g}_2}{\partial \vec{x}}  \\ \vdots \\ \frac{\partial \vec{g}_n}{\partial \vec{x}} \end{bmatrix} \,,
\end{equation} 
where $\vec{g}_i$ is the value of the gradient in $\vec{e}_i$ direction, and can be approximated by
$$ \vec{g}_i = \vec{g}^T \vec{e}_i \approx \sum_{j=1}^n \frac{f(\vec{x}+\vec{d}_j) - f(\vec{x})}{\| \vec{d}_j \|^2} \vec{d}_j^T \vec{e}_i \,.$$
The full gradient of $\vec{g}_i$ can also be approximated using the finite difference method as follows:
$$ \frac{\partial \vec{g}_i}{\partial \vec{x}} \approx \sum_{j=1}^n \frac{\vec{g}_i(\vec{x}+\vec{d}_j) - \vec{g}_i (\vec{x})}{\| \vec{d}_j \|^2} \vec{d}_j^T \,.$$
Now, \eqref{Hess_FinDiff} can be written as
\begin{equation} 
\begin{aligned}
 H = \begin{bmatrix} \frac{\partial \vec{g}_1}{\partial \vec{x}}  \\  \frac{\partial \vec{g}_2}{\partial \vec{x}}  \\ \vdots \\ \frac{\partial \vec{g}_n}{\partial \vec{x}} \end{bmatrix} \approx \begin{bmatrix} \sum_{j=1}^n \frac{\vec{g}_1(\vec{x}+\vec{d}_j) - \vec{g}_1 (\vec{x})}{\| \vec{d}_j \|^2} \vec{d}_j^T \\ \sum_{j=1}^n \frac{\vec{g}_2(\vec{x}+\vec{d}_j) - \vec{g}_2 (\vec{x})}{\| \vec{d}_j \|^2} \vec{d}_j^T  \\ \vdots \\ \sum_{j=1}^n \frac{\vec{g}_n(\vec{x}+\vec{d}_j) - \vec{g}_n (\vec{x})}{\| \vec{d}_j \|^2} \vec{d}_j^T \end{bmatrix} = \sum_{j=1}^n \begin{bmatrix} \vec{g}_1(\vec{x}+\vec{d}_j) - \vec{g}_1 (\vec{x}) \\ \vec{g}_2(\vec{x}+\vec{d}_j) - \vec{g}_2 (\vec{x}) \\ \vdots \\ \vec{g}_n(\vec{x}+\vec{d}_j) - \vec{g}_n (\vec{x}) \end{bmatrix} \frac{\vec{d}_j^T}{\| \vec{d}_j \|^2}\,.
\end{aligned}
\end{equation} 
Thus
$$ H \approx \sum_{j=1}^n \frac{\vec{g}(\vec{x}+\vec{d}_j) - \vec{g}(\vec{x}) }{\| \vec{d}_j \|^2} \vec{d}_j^T = \sum_{i=1}^n \frac{\Delta G_i}{\| \vec{d}_i \|^2} \vec{d}_i^T \,.$$
\end{proof}

\end{document}